\documentclass[12pt]{amsart}
\usepackage{a4wide,enumerate}
\usepackage{color}
\allowdisplaybreaks

\let\pa\partial  
\let\na\nabla  
\let\eps\varepsilon  
\newcommand{\N}{{\mathbb N}}  
\newcommand{\R}{{\mathbb R}} 
\newcommand{\T}{{\mathbb T}} 
\newcommand{\diver}{\operatorname{div}}  

\newcommand{\blue}[1]{\textcolor{black}{#1}}

\newtheorem{theorem}{Theorem}   
\newtheorem{lemma}[theorem]{Lemma}   
   
\newtheorem{remark}[theorem]{Remark}

\begin{document}  

\title[A cross-diffusion system]{A cross-diffusion system derived 
from a Fokker-Planck equation with partial averaging}
\author{Ansgar J\"ungel}
\address{A.J.: Institute for Analysis and Scientific Computing, Vienna University of  
	Technology, Wiedner Hauptstra\ss e 8--10, 1040 Wien, Austria}
\email{juengel@tuwien.ac.at} 

\author{Nicola Zamponi}
\address{N.Z.: Institute for Analysis and Scientific Computing, Vienna University of  
	Technology, Wiedner Hauptstra\ss e 8--10, 1040 Wien, Austria}
\email{nicola.zamponi@tuwien.ac.at}

\date{\today}

\thanks{The authors acknowledge partial support from   
the Austrian Science Fund (FWF), grants P22108, P24304, and W1245, and from
the Austrian-French Project Amade\'e of the Austrian Exchange Service (\"OAD),
grant FR 04/2016.} 

\begin{abstract}
A  cross-diffusion system for two compoments with a Laplacian structure 
is analyzed on the multi-dimen\-sio\-nal 
torus. This system, which was recently suggested by P.-L.~Lions,
is formally derived from a Fokker-Planck equation for the probability
density associated to a multi-dimensional It\={o} process, assuming that
the diffusion coefficients depend on 
partial averages of the probability density with exponential weights.
A main feature is that the diffusion matrix of the limiting cross-diffusion 
system is generally neither symmetric nor positive definite, but its structure 
allows for the use of entropy methods.
The global-in-time existence of positive weak solutions is proved and, 
under a simplifying assumption, the large-time asymptotics is investigated.
\end{abstract}

 

\maketitle


\section{Introduction}

The aim of this paper is the analysis of the following cross-diffusion system
\begin{equation}
  \pa_t u_i = \Delta\big(a(u_1/u_2)u_i\big) + \mu_i u_i, \quad t>0, 
	\quad u_i(0)=u_i^0\ge 0\quad\mbox{in }\T^d,\ i=1,2, \label{1.eq}
\end{equation}
where $\T^d$ is the $d$-dimensional torus with $d\ge 1$, $a:(0,\infty)\to(0,\infty)$
is a continuously differentiable function, and $\mu_i\in\R$.
This system can be formally derived \cite{Lio15} from a $(d+1)$-dimensional
Fokker-Planck equation for the probability density $f(x,y,t)$, where 
$x\in\R^d$, $y\in\R$. The function $u_i$ is obtained from $f$ by partial averaging,
$$
  u_i(x,t) = \int_\R f(x,y,t)e^{\lambda_i y}dy, \quad i=1,2,
$$
$\mu_i$ is a function of $\lambda_i$, 
and $a(u_1/u_2)$ is related to the diffusion coefficients in the
Fokker-Planck equation.
Strictly speaking, equation \eqref{1.eq} holds in $\R^d$ (or on some
subset of $\R^d$) but we consider
this equation on the torus for the sake of simplicity (and to avoid possible
issues with boundary conditions). 
For details on the derivation, we refer to Section \ref{sec.deriv}.

System \eqref{1.eq} has been suggested by P.-L.~Lions in \cite{Lio15}, 
and the global-in-time existence of (weak)
solutions has been identified as an open problem. In this paper, we solve
this problem by applying the entropy method for diffusive equations.

The underlying Fokker-Planck equation for $f(x,y,t)$
models the time evolution of the value of a financial product in an idealized
financial market, depending on various underlying assets or economic values. 
The function $u_i$ is an average with respect
to the variable $y$, which may be interpreted as the value of an economic parameter,
and the exponential weight emphasizes large positive or large negative
values of $y$, depending on the sign of $\lambda_i$. 
We note that partial averaging is also employed to simplify chemical master
equations \cite{MLSH12}. Here, we are not interested
in potential applications, but more in the refinement of mathematical
tools to analyze \eqref{1.eq}. 

We assume that there exist $a_0>0$ and $p\ge 0$ such that for all $r>0$,
\begin{equation}\label{1.ass}
  a(r) \ge r|a'(r)|, \quad a(r) \ge \frac{a_0}{r^p+r^{-p}}.
\end{equation}
The first condition means that $a$ grows at most linearly 
(see Lemma \ref{lem.esta}). The second condition is a technical assumption
needed for the entropy method (see the proof of Lemma \ref{lem.h2A}).
Examples are $a(r)=1$, which leads to uncoupled heat equations for $u_1$ and $u_2$, 
$a(r)=r^\alpha$ with $0<\alpha\le 1$, $a(r)=r^\beta/(1+r^{\beta-1})$ with
$\beta>0$, and $a(r)=1/r$. The last example gives the equations
\begin{equation}\label{1.et}
  \pa_t u_1 = \Delta u_2, \quad \pa_t u_2 = \Delta\bigg(\frac{u_2^2}{u_1}\bigg).
\end{equation}
Surprisingly, this system corresponds (up to a factor) to an energy-transport model
for semiconductors. Indeed, introducing the electron density $n:=u_1$ and
the electron temperature $\theta:=u_2/u_1$, equations \eqref{1.et} can be
written as
$$
  \pa_t n = \Delta(n\theta), \quad \pa_t(n\theta) = \Delta(n\theta^2).
$$
A class of energy-transport models that includes the above example 
was analyzed in \cite{ZaJu15}. 

Another class of models which resembles \eqref{1.eq}
are the equations 
\begin{equation}\label{1.pop}
  \pa_t u_i=\Delta(p_i(u)u_i), \quad i=1,\ldots,m,
\end{equation}
modeling the time evolution
of population densities $u_i$. These systems are analyzed in, e.g., 
\cite{DLMT15,Jue15}, essentially for $m=2$.
In this application, $p_i$ is often given by the sum $p_{i1}(u_1)+p_{i2}(u_2)$,
and consequently, the results of \cite{DLMT15,Jue15} do not apply and we need
to develop new ideas.

Our first main result is the global-in-time existence of weak solutions to
\eqref{1.eq}. 

\begin{theorem}[Existence of weak solutions]\label{thm.ex}
Let \eqref{1.ass} hold and let $T>0$, $\alpha\ge p+4$, \blue{$\mu_1$, $\mu_2\in\R$},
\blue{$0\le a\in C^1(0,\infty)$},
$u^0=(u_1^0,u_2^0)\in L^2(\T^d)^2$ with $u_1^0$, $u_2^0\ge 0$ in $\T^d$ and 
\blue{$H[u^0]<\infty$}.
Then there exists a solution $u=(u_1,u_2)$ to \eqref{1.eq} satisfying
$u_i>0$ in $\T^d$, $t>0$, $i=1,2$, and
\begin{align*}
  & u_i,\,a(u_1/u_2)u_i\in L^\infty(0,T;L^2(\T^d)),  \\
	& \na u_i,\,\na\big(a(u_1/u_2)u_i\big)\in	L^2(0,T;L^2(\T^d)), \quad
	\pa_t u_i\in L^2(0,T;H^1(\T^d)'),\quad i=1,2.
\end{align*}
If additionally $\mu_i\le 0$ for $i=1,2$, we have the uniform bounds
\begin{equation}\label{1.unif}
  u_i,\,a(u_1/u_2)u_i\in L^\infty(0,\infty;L^2(\T^d)), \quad
	\na u_i,\,\na\big(a(u_1/u_2)u_i\big)\in	L^2(0,\infty;L^2(\T^d)).
\end{equation}
\end{theorem}

As mentioned above, the proof of this theorem is based on entropy methods.
These methods have been originally developed to understand the large-time behavior
of solutions; see, e.g., \cite{AMTU01,Vil09}. The ``entropy'' of system
\eqref{1.eq} is often understood as a convex Lyapunov functional which
provides suitable nonlinear gradient estimates. In many situations, and also
in the financial context presented here, the ``entropy'' has no physical
counterpart. However, we claim that this notion is appropriate since
it naturally generalizes physical situations. For details, we refer to \cite{Jue15}.

Our key idea is to employ the functional
\begin{equation}\label{1.H}
  H[u] = \int_{\T^d}h(u)dx, \quad
	h(u) = \bigg(\frac{u_1}{u_2}\bigg)^\alpha u_1^2 
	+ \bigg(\frac{u_1}{u_2}\bigg)^{-\alpha}u_2^2 + u_1 - \log u_1 + u_2 - \log u_2,
\end{equation}
where $\alpha\ge p+4$ and $u=(u_1,u_2)\in(0,\infty)^2$. We will show that 
\begin{equation}\label{1.ei}
  \frac{d}{dt}H[u] + \int_{\T^d}\bigg(\bigg(\frac{u_1}{u_2}\bigg)^{\alpha-p}
	+ \bigg(\frac{u_1}{u_2}\bigg)^{p-\alpha}\bigg)\big(|\na u_1|^2+|\na u_2|^2\big)dx
	\le CH[u]
\end{equation}
for some constant $C>0$ which vanishes if $\mu_1=\mu_2=0$.
In this situation, the mapping $t\mapsto H[u(t)]$ is nonincreasing;
otherwise, for $\mu_i\neq 0$, $t\mapsto H[u(t)]$ is bounded on finite time intervals. 
We infer from the inequality $x+x^{-1}\ge 2$ for all $x>0$ uniform bounds
for $u_i(t)$ in $H^1(\T^d)$, which are needed for the compactness argument.

The entropy method gives more than just the a priori estimate \eqref{1.ei}.
Indeed, let us write \eqref{1.eq} in divergence form:
$$
  \pa_t u - \diver(A(u)\na u) = f(u), \quad t>0, \quad u(0)=u^0\quad\mbox{in }\T^d,
$$
where the $i$th component of
$\diver(A(u)\na u)$ equals $\sum_{j=1}^d\sum_{k=1}^2\pa_j(A_{ik}(u)\pa_j u_k)$,
$\pa_j=\pa/\pa x_j$, and $f(u)=(\mu_1 u_1,\mu_2 u_2)^\top$. The diffusion matrix
\begin{equation}\label{1.A}
  A(u) = \begin{pmatrix}
	a(u_1/u_2) + (u_1/u_2)a'(u_1/u_2) & -(u_1/u_2)^2 a'(u_1/u_2) \\
	a'(u_1/u_2) & a(u_1/u_2) - (u_1/u_2)a'(u_1/u_2)
	\end{pmatrix}
\end{equation}
is generally neither symmetric nor positive definite. 
Since the only eigenvalue of $A(u)$ is given by $\lambda=a(u_1/u_2)>0$, the
system is normally elliptic \cite{Ama93} and local-in-time existence of
classical solutions can be expected. The difficulty is to prove the
global-in-time existence. The entropy density $h(u)$ allows us to 
formulate \eqref{1.eq} in new variables with a positive semidefinite diffusion
matrix. Then, together with the a priori estimates from \eqref{1.ei},
global existence will be deduced.
Indeed, defining the so-called entropy variable $w=(w_1,w_2)$ by
$w_i=\pa h/\pa u_i$ ($i=1,2$), equation \eqref{1.eq} is equivalent to
\begin{equation}\label{1.eqw}
  \pa_t u - \diver(B(w)\na w) = f(u), \quad t>0, \quad u(0)=u^0\quad\mbox{in }\T^d,
\end{equation}
where $B(w)=A(u)h''(u)^{-1}$ is positive semidefinite (see Lemma \ref{lem.h2A})
and $h''(u)$ is the Hessian matrix of $h(u)$. With this formulation, we obtain
$$
  \frac{d}{dt}H[u] + \int_{\T^d}\na u:h''(u)A(u)\na u dx = \int_{\T^d}f(u)\cdot w dx,
$$
where $A:B=\sum_{j=1}^d\sum_{k=1}^2 A_{kj}B_{kj}$ for two matrices 
$A=(A_{kj})$, $B=(B_{kj})\in\R^{2\times d}$.
The right-hand side can be bounded in terms of $H[u]$ (see \eqref{2.h}),
and the integral on the left-hand side is related to the corresponding
integral in \eqref{1.ei}. 

The proof of Theorem \ref{thm.ex} is based on a regularization of \eqref{1.eqw},
the fixed-point theorem of Leray-Schauder, and the de-regularization limit.
The compactness is obtained from the entropy estimate \eqref{1.ei}. This
technique is similar to those employed in our works \cite{Jue15,ZaJu15}.
The novelty here is the (nontrivial) observation 
that the cross-diffusion system \eqref{1.eq}
possesses a convex Lyapunov functional, defined by \eqref{1.H}. 
Moreover, compared to \cite{Jue15,ZaJu15}, we are facing additional 
technical difficulties due to the quotient $u_1/u_2$.

The second result concerns the large-time asymptotics in the case $\mu_i=0$
for $i=1,2$. 

\begin{theorem}[Large-time asymptotics]\label{thm.time}
Let the assumptions of Theorem \ref{thm.ex} hold and let 
$\mu_1=\mu_2=0$.
Then the solution $u(t)=(u_1,u_2)(t)$ to \eqref{1.eq} converges in $L^2(\T^d)$ to
$\overline{u}=(\overline{u}_1,\overline{u}_2)$ as $t\to\infty$, where
$$
  \overline{u}_i = \frac{1}{\mathrm{meas}(\T^d)}\int_{\T^d}u_i^0dx, \quad i=1,2.
$$
\end{theorem} 

If $\mu_i<0$ for $i=1,2$, 
we prove the exponential convergence of $u(t)$ to zero
in $H^1(\T^d)'$, see Remark \ref{rem.conv}. For a discussion of the case 
$\mu_i>0$, we refer to Remark \ref{rem.conv2}.

The paper is organized as follows. In Section \ref{sec.deriv}, we make
precise the derivation of \eqref{1.eq} from a Fokker-Planck equation.
Some technical results are proved in Section \ref{sec.aux}. Section \ref{sec.ex}
is devoted to the proof of Theorem \ref{thm.ex}, and Theorem \ref{thm.time}
is shown in Section \ref{sec.time}.


\section{Derivation of the cross-diffusion system \eqref{1.eq}}\label{sec.deriv}

We summarize the formal derivation of \eqref{1.eq} from a Fokker-Planck equation
as presented by P.-L.~Lions in \cite{Lio15}.
Consider the $n$-dimensional It\={o} process $X_t=(X_t^1,\ldots,X_t^n)$ 
on some probability space,
driven by the $n$-dimen\-sion\-al Wiener process $W_t=(W^1_t,\ldots,W_t^n)$
with respect to some given filtration. 
We assume that $X_t$ solves the stochastic differential equation
$$
  dX_t = \widetilde\mu_t(X_t) dt + \sigma_t(X_t)dW_t, \quad t>0,
$$
where $\widetilde\mu_t=(\widetilde\mu^1_t,\ldots,\widetilde\mu_t^n)$, and 
$\sigma_t=(\sigma^{ij}_t)_{i,j,=1,\ldots,n}$ is an $n\times n$ matrix.
It is well known \cite[Theorems 7.3.3, 8.2.1]{Oek03} that
the probability density $f(x_1,\ldots,x_n,t)$ for $X_t$ satisfies
the Fokker-Planck (or forward Kolmogorov) equation
$$
  \pa_t f = \frac12\sum_{i,j=1}^n\frac{\pa^2}{\pa x_i\pa x_j}(D_{ij}(\widehat x)f)
	- \sum_{i=1}^n\frac{\pa}{\pa x_i}(\widetilde\mu^i(\widehat x)f), 
	\quad \widehat x\in\R^n,\ t>0,
$$
where $D(\widehat x)=(D_{ij}(\widehat x))=\sigma(\widehat x)\sigma(\widehat x)^\top$ 
is the diffusion tensor and $\widehat x=(x_1,\ldots,x_n)$.

In the following, we set $\widetilde\mu_t=0$ and 
$\sigma_t=\mbox{diag}(\sigma_1,\ldots,\sigma_n)$.
This means that we neglect correlations between the processes.
Taking them into account will lead to first-order terms in the final equations;
see Remark \ref{rem.gen}.
Under the above simplifications, the Fokker-Planck equation becomes
\begin{equation}\label{2.fps}
  \pa_t f = \frac12\sum_{j=1}^n\frac{\pa^2}{\pa x_j^2}(\sigma_j^2 f), \quad
	\widehat x\in\R^n,\ t>0.
\end{equation}
We assume that $\sigma_j$ is a function of the partial averages
$$
  u_i(x,t) = \int_\R f(x,x_n,t)e^{\lambda_i x_n}dx_n, \quad x=(x_1,\ldots,x_{n-1}),\
	i=1,\ldots,m,
$$
where $\lambda_i$ are some given (pairwise different) parameters.
Temporal averages appear, for instance, in the modeling of Asian options.
Here, $u_i$ may be interpreted as an average with respect to the ecocnomic parameter
$x_n$. We may employ other weights than the exponential one but this one
is mathematically extremely convenient because of the property
$\pa u_i/\pa x_n=\lambda_i u_i$ (see Remark \ref{rem.gen}).
Multiplying \eqref{2.fps} by $e^{\lambda_i x_n}$ and integrating with respect
to $x_n\in\R$, a straightforward calculation shows that $u_i$ solves
\begin{equation}\label{der.ui}
  \pa_t u_i = \frac12\sum_{j=1}^{n-1}\frac{\pa^2}{\pa x_j^2}(\sigma_j^2u_i)
	+ \frac{\lambda_i^2}{2}\sigma_n^2 u_i, \quad i=1,\ldots,m.
\end{equation}
We allow $\sigma_j$ to depend on the partial averages, 
$\sigma_j=\sigma_j(u_1,\ldots,u_m)$.

We consider only the special case $m=2$, $\sigma:=\sigma_j$ for $j=1,\ldots,n-1$,
and $\sigma_n$ is constant and positive.
Setting $u=(u_1,u_2)$, $\mu_i:=\lambda_i^2\sigma_n/2$, we find that
\begin{equation}\label{der.ui2}
  \pa_t u_i = \frac12\Delta(\sigma(u)^2u_i) + \mu_i u_i, \quad 
	x\in\R^{n-1},\ t>0,\ i=1,2.
\end{equation}
In divergence form, this system is equivalent to
$$
  \pa_t u = \diver(A(u)\na u), \quad\mbox{where }
	A(u) = \sigma\begin{pmatrix}
	\sigma + 2\pa_1\sigma u_1 & 2\pa_2\sigma u_1 \\
	2\pa_1\sigma u_2 & \sigma + 2\pa_2\sigma u_2
	\end{pmatrix},
$$
where $\pa_i\sigma=\pa\sigma/\pa u_i$, $i=1,2$.
This system is of parabolic type in the sense of Petrovski if the real parts
of the eigenvalues of $A$ are nonnegative \cite{Ama93}, i.e.\ if
$\sigma + \pa_1\sigma u_1 + \pa_2\sigma u_2\ge 0$ for all $u\in\R^2$.
This requirement is fulfilled if, for instance, $\sigma$ depends on the quotient
$u_1/u_2$ only. Therefore, we set $\sigma(u)^2=2a(u_1/u_2)$. Then
$$
  \pa_t u_i = \Delta(a(u_1/u_2)u_i) + \mu_i u_i, \quad x\in\R^{n-1},\ t>0,\ i=1,2,
$$
is of parabolic type in the sense of Petrovski, and these equations correspond to
\eqref{1.eq}.

\begin{remark}[Generalizations]\label{rem.gen}\rm
The general model for nonvanishing $\widetilde\mu^i_t$ and nondiagonal $\sigma_t$
is derived as above, and the result reads as
\begin{equation}\label{der.ui3}
  \pa_t u_i = \frac12\sum_{j,k=1}^{n-1}\frac{\pa^2}{\pa x_j\pa x_k}(D_{jk} u_i)
	- \frac12\sum_{j=1}^{n-1}\frac{\pa}{\pa x_j}\big(
	(2\widetilde\mu^j + \lambda_i D_{jn})u_i\big)
	+ \frac{\lambda_i}{2}(2\widetilde\mu^n+\lambda_i D_{nn})u_i.
\end{equation}
Compared to \eqref{der.ui}, this equation also contains first-order terms.
If $\widetilde\mu^i_t=0$ and $\sigma_t$ is diagonal,
we obtain $m$ equations of the type \eqref{der.ui2}. 
The analysis of \blue{cross-diffusion} systems with more than two
components is expected to be much more involved than for those with two components.
For instance, the analysis of the cross-diffusion model \eqref{1.pop}
is rather well understood only in the case of $m=2$ components, 
while the case of $m\ge 3$ equations
requires additional properties \cite{DaJu16}.

Another generalization concerns nonexponential weights. For instance, 
we may define
$$
  u_i = \int_\R f(x,t)\sin(\lambda_i x_n)dx_n, \quad i=1,\ldots,m.
$$
Choosing again $\widetilde\mu^i_t=0$ and 
$\sigma_t=\mbox{diag}(\sigma_1,\ldots,\sigma_n)$, we find that
$$
  \pa_t u_i = \frac12\sum_{j=1}^{n-1}\frac{\pa^2}{\pa x_j^2}(\sigma_j(u)^2u_i)
	- \frac{\lambda_i^2}{2}\sigma_n^2 u_i, \quad i=1,\ldots,m.
$$
This justifies the assumption $\mu_i\in\R$ in \eqref{1.eq} but there seems
to be no financial interpretation of the trigonometric weight functions.
\qed
\end{remark}


\section{Some auxiliary lemmas}\label{sec.aux}

In this section, we prove some algebraic properties of the matrices
$h''(u)$ and $A(u)$ and some estimates related to the entropy density $h(u)$
and the components of $A(u)$. Recall that $h(u)$ is defined in \eqref{1.H}
and $A(u)$ in \eqref{1.A}.

\begin{lemma}[Properties of $h$]\label{lem.h}
Let $\alpha>0$. The function $h:(0,\infty)^2\to\R^2$, defined in \eqref{1.H}, 
is convex, its derivative $h'$ is invertible, and there exists $C_h>0$ such that 
for all $u=(u_1,u_2)\in(0,\infty)^2$,
\begin{equation}\label{2.h}
  h(u)\ge\frac{1}{2}( u_1^2+u_2^2 ), \quad
	\sum_{i=1}^2\mu_i u_i\pa_ih(u)\le C_h h(u),
\end{equation}
where we recall that $\pa_ih=\pa h/\pa u_i$.
\end{lemma}

\begin{proof}
We proceed in several steps.

{\em Step 1: $h$ is convex.}
We compute the first partial derivatives of $h$,
\begin{align}
  \pa_1 h(u) &= (\alpha+2)u_1^{\alpha+1}u_2^{-\alpha} 
	- \alpha u_1^{-\alpha-1}u_2^{\alpha+2}
	- u_1^{-1} + 1, \label{ex.hu} \\
	\pa_2 h(u) &= (\alpha+2)u_1^{-\alpha}u_2^{\alpha+1} 
	- \alpha u_1^{\alpha+2}u_2^{-\alpha-1}
	- u_2^{-1} + 1, \label{ex.hv}
\end{align}
and the Hessian $h''(u)=H^{(1)}+H^{(2)}+H^{(3)}$, where
\begin{align}
  H^{(1)} &=
  \begin{pmatrix}
  (\alpha+2)(\alpha+1)(u_1/u_2)^{\alpha} & -\alpha(\alpha+2)(u_1/u_2)^{\alpha+1}\\
  -\alpha(\alpha+2)(u_1/u_2)^{\alpha+1} & \alpha(\alpha+1)(u_1/u_2)^{\alpha+2}
  \end{pmatrix}, \nonumber \\
  H^{(2)} &= 
  \begin{pmatrix}
  \alpha(\alpha+1)(u_2/u_1)^{\alpha+2} & -\alpha(\alpha+2)(u_2/u_1)^{\alpha+1}\\
  -\alpha(\alpha+2)(u_2/u_1)^{\alpha+1} & (\alpha+2)(\alpha+1)(u_2/u_1)^\alpha
  \end{pmatrix}, \label{ex.h2A} \\
  H^{(3)} &= 
  \begin{pmatrix}
  u_1^{-2} & 0\\ 0 & u_2^{-2}
  \end{pmatrix}. \nonumber
\end{align}
Since $\det H^{(1)}=\alpha(\alpha+2)(u_1/u_2)^{2(\alpha+1)}>0$,
$\det H^{(2)}=\alpha(\alpha+2)(u_2/u_1)^{2(\alpha+1)}>0$, and the diagonal elements
of $H^{(1)}$, $H^{(2)}$ are positive, the matrices $H^{(i)}$, $i=1,2,3$,
are positive definite and so does $h''(u)$. Thus, $h$ is convex.

{\em Step 2: $h'$ is invertible.}
Since the Hessian $h''$ is \blue{positive definite} 
on $(0,\infty)^2$, $h'$ is one-to-one
and the image $R(h')$ is open. If $R(h')$ is also closed, it follows that
$R(h')=\R^2$ which means that $h'$ is surjective. For this, let
$(w_n)\in R(h')$ for $n\in\N$ such that $w_n\to w$ as
$n\to\infty$. We show that $w\in R(h')$. By definition, there
exists $u_n>0$ such that $w_n=h'(u_n)$ for
$n\in\N$. The idea is to prove that $(u_n)=(u_{1,n},u_{2,n})$ is a bounded 
and strictly positive sequence.
This implies that, up to a subsequence, $u_n\to u\in(0,\infty)^2$ 
as $n\to\infty$. By continuity of $h'$, we infer that $h'(u_n)\to h'(u)$
as $n\to\infty$. We already know that $h'(u_n)=w_n\to w$ which
shows that $w=h'(u)\in R(h')$, and $R(h')$ is closed.

It remains to verify that there exist positive constants $m$, $M>0$ such that
$m\le u_{i,n}\le M$ for all $n\in\N$, $i=1,2$. We argue by contradiction.
Let us assume that (up to a subsequence) $u_{1,n}\to 0$ as $n\to\infty$. 
Since $(w_{1,n})=(\pa_1h(u_n))$ is convergent, we deduce from \eqref{ex.hu} that
$u_{2,n}\to 0$ as well. As a consequence,
$$
  \alpha u_{1,n} w_{1,n} + (\alpha+2)u_{2,n} w_{2,n}\to 0, \quad
	(\alpha+2)u_{1,n} w_{1,n} + \alpha u_{2,n} w_{2,n}\to 0.
$$
Expanding these expressions yields 
$$
  u_{1,n}^{-\alpha}u_{2,n}^{\alpha+2} \to \frac12, \quad
	u_{1,n}^{\alpha+2}u_{2,n}^{-\alpha} \to \frac12,
$$
and the product also converges, $u_{1,n}^2u_{2,n}^2\to 1/4$.
This is absurd since $(u_n)$ converges to zero. Therefore, $u_{1,n}$
is strictly positive. With an analogous argument, we conclude that $u_{2,n}$ is
strictly positive too.

Let us assume that (up to a subsequence) $u_{1,n}\to\infty$ as $n\to\infty$. 
Again, the convergence
of $(w_{1,n})$ and \eqref{ex.hu} imply that $u_{2,n}\to\infty$. Consequently,
$$
  \frac{\alpha}{u_{2,n}}w_{1,n} + \frac{\alpha+2}{u_{1,n}}w_{2,n} \to 0, \quad
	\frac{\alpha+2}{u_{2,n}}w_{1,n} + \frac{\alpha}{u_{1,n}}w_{2,n}\to 0,
$$
from which we infer after expanding these expressions that 
$u_{2,n}/u_{1,n}\to 0$ and $u_{1,n}/u_{2,n}\to 0$, 
which is a contradiction.
So, $(u_{1,n})$ is bounded, and the same conclusion holds for $(u_{2,n})$.

{\em Step 3: proof of \eqref{2.h}.}
Observing that $x-\log x\ge 1$ for all $x>0$, it follows that
$$
  h(u) \ge u_1^2\bigg(\bigg(\frac{u_1}{u_2}\bigg)^\alpha 
	+ \bigg(\frac{u_2}{u_1}\bigg)^{\alpha+2}\bigg)
	= u_2^2\bigg(\bigg(\frac{u_1}{u_2}\bigg)^{\alpha+2} 
	+ \bigg(\frac{u_2}{u_1}\bigg)^{\alpha}\bigg).
$$
The elementary inequality \blue{$x^\alpha+(1/x)^{\alpha+2}\ge 1$} for $x>0$ shows 
the first inequality in \eqref{2.h}:
$$
  h(u) \ge \frac{u_1^2}{2}\bigg(\bigg(\frac{u_1}{u_2}\bigg)^\alpha 
	+ \bigg(\frac{u_2}{u_1}\bigg)^{\alpha+2}\bigg)
	+ \frac{u_2^2}{2}\bigg(\bigg(\frac{u_1}{u_2}\bigg)^{\alpha+2} 
	+ \bigg(\frac{u_2}{u_1}\bigg)^{\alpha}\bigg)
	\ge\frac{1}{2}( u_1^2+u_2^2 ).
$$
For the second inequality in \eqref{2.h}, we employ
definition \eqref{1.H} of $h$ and the elementary inequality
$x-1\le 2(x-\log x)$ for $x>0$ to find that, 
\blue{if $C_h = 2(\alpha+2)(|\mu_1|+|\mu_2|)$},
\begin{align*}
  \sum_{i=1}^2\mu_i u_i\pa_i h(u)
	&= \big(\mu_1(\alpha+2)-\mu_2\alpha\big)u_1^{\alpha+2}u_2^{-\alpha}
	+ \big(\mu_2(\alpha+2)-\mu_1\alpha\big)u_1^{-\alpha}u_2^{\alpha+2} \\
	&\phantom{xx}{}+ \mu_1(u_1-1) + \mu_2(u_2-1) \\
	&\le C_h(u_1^{\alpha+2}u_2^{-\alpha} + u_1^{-\alpha}u_2^{\alpha+2}) 
	+ C_h(u_1-\log u_1+u_2-\log u_2).
\end{align*}
This finishes the proof.
\end{proof}

Next, we prove that $h''(u)A)(u)$ is positive semidefinite.
Then $B=A(u)h''(u)^{-1}$ in \eqref{1.eqw} is positive semidefinite too,
since $z^\top A(u)h''(u)^{-1}z$ $=(h''(u)^{-1}z)^\top h''(u)A(u)(h''(u)^{-1}z)\ge 0$
for $z\in\R^2$.

\begin{lemma}[Positive semidefiniteness of $h''A$]\label{lem.h2A}
Let condition \eqref{1.ass} hold.
If $\alpha(\alpha+2)>1$, the matrix $h''(u)A(u)$ is positive semidefinite in
$(0,\infty)^2$. Furthermore, if additionally $\alpha\ge p$, there exists
a constant $\kappa=\kappa(\alpha)>0$ such that for all $u=(u_1,u_2)\in(0,\infty)^2$
and $z\in\R^2$,
$$
  z^\top h''(u)A(u)z \ge \kappa\bigg(\bigg(\frac{u_1}{u_2}\bigg)^{\alpha-p}
	+ \bigg(\frac{u_1}{u_2}\bigg)^{p-\alpha}\bigg)|z|^2.
$$
\end{lemma}

\begin{proof}
Let $\alpha(\alpha+2)>1$ and let $M^{(i)}=(M_{jk}^{(i)})
:=\frac12((H^{(i)}A)^\top+H^{(i)}A)$
be the symmetric part of $H^{(i)}A$, where $H^{(i)}$ with $i=1,2,3$
is defined in \eqref{ex.h2A}. A computation shows that 
\begin{align*}
   M_{11}^{(1)} &= (\alpha + 2)\big((\alpha + 1)a(u_1/u_2) + (u_1/u_2)a'(u_1/u_2)\big)
	(u_1/u_2)^\alpha, \\
  \det M^{(1)} &= \big(\alpha(\alpha + 2)a(u_1/u_2)^2 - (u_1/u_2)^2 a'(u_1/u_2)^2\big)
	(u_1/u_2)^{2\alpha+2}, \\
  M_{11}^{(2)} &= \alpha\big((\alpha + 1)a(u_1/u_2) - (u_1/u_2)a'(u_1/u_2)\big)
	(u_2/u_1)^{\alpha+2}, \\
  \det M^{(2)} &= \big(\alpha(\alpha + 2)a(u_1/u_2)^2 - (u_1/u_2)^2 a'(u_1/u_2)^2)
	(u_2/u_1)^{2\alpha+2}, \\
  M^{(3)} &=
  \begin{pmatrix}
  (a(u_1/u_2)+(u/u_2)a'(u_1/u_2))u_1^{-2} & 0 \\
  0 & (a(u_1/u_2)-(u_1/u_2)a'(u_1/u_2))u_2^{-2}
  \end{pmatrix}.
\end{align*}
By the first condition in \eqref{1.ass} and the positivity of $\alpha$, we infer that
$M^{(3)}$ is positive semidefinite and $M_{11}^{(1)}$, $M_{11}^{(2)}$ are
positive for $u$, $v>0$. Moreover, since $\alpha(\alpha+2)>1$ by assumption,
$\det(M^{(1)})>0$ and $\det(M^{(2)})>0$. Thus,
\blue{by Sylvester's criterion}, $(h''A)(u)$ is positive
semidefinite for all $u\in(0,\infty)^2$. 

Now let additionally $\alpha\ge p$. 
Then the first condition in \eqref{1.ass} shows that
\begin{align*}
  \frac{\det M^{(1)}}{\operatorname{tr} M^{(1)}}
	&= \frac{\alpha(\alpha + 2)a(u_1/u_2)^2 - (u_1/u_2)^2 a'(u_1/u_2)^2}{
  (\alpha (u_1/u_2)^2 + \alpha + 2)((\alpha + 1)a(u_1/u_2) + (u_1/u_2) a'(u_1/u_2))}
  (u_1/u_2)^{\alpha+2} \\
	&\ge \frac{(\alpha(\alpha+2)-1)a(u_1/u_2)^2}{(\alpha+2)((u_1/u_2)^2+1)
	(\alpha+2)a(u_1/u_2)}(u_1/u_2)^{\alpha+2} \\
	&= k_1(\alpha)\frac{a(u_1/u_2)}{(u_1/u_2)^2+1}(u_1/u_2)^{\alpha+2},
\end{align*}
where $k(\alpha)=(\alpha(\alpha+2)-1)/(\alpha+2)^2$. In a similar way, we find that
\begin{align*}
  \frac{\det M^{(2)}}{\operatorname{tr} M^{(2)}}
  &= \frac{\alpha(\alpha+2)a(u_1/u_2)^2 - (u_1/u_2)^2a'(u_1/u_2)^2}{
	((\alpha+2)(u_1/u_2)^2+\alpha)((\alpha+1)a(u_1/u_2)-(u_1/u_2)a'(u_1/u_2))}
	(u_2/u_1)^\alpha \\
  &\ge \frac{(\alpha(\alpha+2)-1)a(u_1/u_2)^2}{(\alpha+2)((u_1/u_2)^2+1)
	(\alpha+2)a(u_1/u_2)}(u_1/u_2)^{-\alpha} \\
	&= k(\alpha)\frac{a(u_1/u_2)}{(u_1/u_2)^2+1}(u_1/u_2)^{-\alpha}.
\end{align*}
Since $\det M/\operatorname{tr}M$ is a lower bound for the eigenvalues of any
symmetric positive definite matrix $M\in\R^{2\times 2}$ 
(and taking into account that
$M^{(3)}$ is positive definite), we deduce that for $z\in\R^2$,
\begin{align*}
  z^\top(h''A)(u)z 
	&\ge k(\alpha)a(u_1/u_2)\frac{(u_1/u_2)^{\alpha+2}+(u_1/u_2)^{-\alpha}}{
	(u_1/u_2)^2+1}|z|^2 \\
	&\ge \frac12k(\alpha)a(u_1/u_2)((u_1/u_2)^\alpha + (u_1/u_2)^{-\alpha})|z|^2.
\end{align*}
In the last inequality, we have employed the elementary inequality
$(x^{\alpha+2}+x^{-\alpha})/(x^2+1)\ge \frac12(x^\alpha + x^{-\alpha})$
which is equivalent to $(x^2-1)(x^\alpha-x^{-\alpha})\ge 0$, and this
holds true for all $x>0$. By the second condition in \eqref{1.ass},
$$
  z^\top(h''A)(u)z \ge \frac{a_0}{2} k(\alpha)\frac{(u_1/u_2)^\alpha 
	+ (u_1/u_2)^{-\alpha}}{(u_1/u_2)^p + (u_1/u_2)^{-p}}|z|^2.
$$
The inequality $(x^\alpha+x^{-\alpha})/(x^p+x^{-p})
\ge \frac12(x^{\alpha-p}+x^{p-\alpha})$ is equivalent to
$(x^{\alpha-p}-x^{p-\alpha})(x^p-x^{-p})\ge 0$, which holds true
for $x>0$ since $\alpha-p\ge 0$ and $p\ge 0$. Therefore,
$$
  z^\top(h''A)(u)z \ge \frac{a_0}{4} k(\alpha)
	\big((u_1/u_2)^{\alpha-p} + (u_1/u_2)^{p-\alpha})\big)|z|^2,
$$
which concludes the proof with $\kappa=a_0 k(\alpha)/4$.
\end{proof}

The following two lemmas concern elementary estimates for $a(r)$.

\blue{
\begin{lemma}\label{lem.esta}
Given $r_0>0$ arbitrary, it holds that
$$
  a(r) \le \left\{\begin{array}{ll}
	\frac{a(r_0)}{r_0}r &\mbox{for }r\ge r_0, \\
	r_0 a(r_0)\frac{1}{r} &\mbox{for }r<r_0.
  \end{array}\right.
$$
\end{lemma}
}

\blue{
\begin{proof}
The first inequality in \eqref{1.ass} implies that $r\mapsto a(r)/r$ is
nonincreasing, while $r\mapsto a(r)r$ is nondecreasing. Writing these
monotonicity properties in an explicit way gives the result.
\end{proof}
}

\blue{
\begin{lemma}\label{lem.estal}
Let $\alpha\ge 2$. Then, for all $u_1$, $u_2>0$,
\begin{align*}
  a\bigg(\frac{u_1}{u_2}\bigg)^2(u_1^2 + u_2^2)
	&\le C_a\bigg(u_1^2 + u_2^2 + \frac{u_1^4}{u_2^2}\bigg) \\
	&\le \xi_\alpha C_a\bigg(\bigg(\frac{u_1}{u_2}\bigg)^\alpha u_1^2 
	+ \bigg(\frac{u_1}{u_2}\bigg)^{-\alpha}u_2^2\bigg) 
	\le \xi_\alpha C_a h(u),
\end{align*}
where \blue{$C_a=a(1)^2$ and $\xi_\alpha>0$ is a suitable constant which
only depends on $\alpha$.}
\end{lemma}
}
\blue{
\begin{proof}
The first inequality follows from an application of Lemma \ref{lem.esta} with $r_0=1$.
Indeed, if $u_1/u_2 \ge 1$, we obtain
$$
  a\bigg(\frac{u_1}{u_2}\bigg)^2 u_1^2 \le a(1)^2\frac{u_1^2}{u_2^2}u_1^2, \quad
	a\bigg(\frac{u_1}{u_2}\bigg)^2u_2^2 \le a(1)^2 u_1^2,
$$
while if $u_1/u_2\le 1$ we have
$$
   a\bigg(\frac{u_1}{u_2}\bigg)^2 u_1^2 \le a(1)^2 u_2^2, \quad
	a\bigg(\frac{u_1}{u_2}\bigg)^2 u_2^2 \le a(1)^2 \frac{u_2^2}{u_1^2} u_2^2.
$$
These inequalities show the claim with $C_a = a(1)^2$. The second inequality
follows from
$$
  u_1^2 + u_2^2 + \frac{u_1^4}{u_2^2} + \frac{u_2^4}{u_1^2} 
  = u_1 u_2 \left( \frac{u_1}{u_2} + \frac{u_2}{u_1} + \frac{u_1^3}{u_2^3} 
	+ \frac{u_2^3}{u_1^3} \right)
  \leq \xi_\alpha u_1 u_2 \left( \frac{u_1^{\alpha+1}}{u_2^{\alpha+1}} 
	+ \frac{u_2^{\alpha+1}}{u_1^{\alpha+1}} \right),
$$
where $\xi_\alpha>0$ is a suitable constant, which depends only on $\alpha$.
This finishes the proof.
\end{proof}
}

\begin{lemma}\label{lem.estA}
Recall that $A(u)=(A_{ij}(u))$ is given by \eqref{1.A}. Then there exists
$C_A>0$, only depending $\max_{0\le r\le 1}a(r)$, 
such that for all $u_1$, $u_2>0$,
$$
  |A(u)|\le C_A\bigg(1 + \bigg(\frac{u_1}{u_2}\bigg)^2
	+ \bigg(\frac{u_1}{u_2}\bigg)^{-2}\bigg).
$$
\end{lemma}

\blue{
\begin{proof}
We apply the first condition in \eqref{1.ass} 
and Lemma \ref{lem.esta} with $r_0=1$ to find that
$$
  \sum_{i,j=1}^2|A_{ij}(u)|\le a\bigg(\frac{u_1}{u_2}\bigg)\bigg(4 + \frac{u_1}{u_2} 
	+ \frac{u_2}{u_1}\bigg).
$$
Then Lemma \ref{lem.esta} with $r_0=1$ implies that
$$
  \sum_{i,j=1}^2|A_{ij}(u)|\le a(1)\left(\frac{u_1}{u_2} 
	+ \frac{u_2}{u_1}\right)\left(4 + \frac{u_1}{u_2} + \frac{u_2}{u_1}\right).
$$
This estimate and Young's inequality conclude the proof.
\end{proof}
}


\section{Proof of Theorem \ref{thm.ex}}\label{sec.ex}

Let $T>0$, $N\in\N$, $\tau=T/N$, and $m\in\N$ with $m>d/2$. Then
the embedding $H^m(\T^d)\hookrightarrow L^\infty(\T^d)$ is compact. Furthermore,
let $w^{k-1}=(w_1^{k-1},w_2^{k-1})\in L^\infty(\T^d)^2$ be given and let
$u^{k-1}=(h')^{-1}(w^{k-1})$. 
By Lemma \ref{lem.h}, the pair $u^{k-1}=(u_1^{k-1},u_2^{k-1})$ is well defined 
and we have $u^{k-1}\in L^\infty(\T^d)^2$.
We wish to find $w_k=(w_1^k,w_2^k)\in H^m(\T^d)^2$ such that for all
$\phi=(\phi_1,\phi_2)\in H^m(\T^d)^2$, 
\begin{align}
  \frac{1}{\tau} \int_{\T^d}(u^k-u^{k-1})\cdot\phi dx 
	&+ \int_{\T^d}\na\phi: B(w^k)\na w^k dx \nonumber \\
	&{}+ \tau\int_{\T^d}(D^m w^k\cdot D^m\phi + w^k\cdot\phi)dx 
	= \sum_{i=1}^2\mu_i\int_{\T^d}u_i^k\phi_i dx, \label{ex.semi} 
\end{align}
where $B(w^k)=A(u^k)h''(u^k)^{-1}$,
$$
  D^m w^k\cdot D^m\phi := \sum_{|\alpha|=m}\sum_{i=1}^2 D^\alpha u_i^k D^\alpha\phi_i,
$$
$\alpha=(\alpha_1,\ldots,\alpha_d)\in\N_0^d$ is a multiindex and 
$D^\alpha=\pa^{|\alpha|}/(\pa x_1^{\alpha_1}\cdots\pa x_d^{\alpha_d})$ 
a partial derivative of order $|\alpha|$. 

{\em Step 1: solution of \eqref{ex.semi}.}
Let $\widehat w=(\widehat w_1,\widehat w_2)\in L^\infty(\T^d)^2$ and 
$\eta\in[0,1]$ be given.
Set $\widehat u=(\widehat u_1,\widehat u_2):=(h')^{-1}(\widehat w)$. 
We solve first the linear problem
\begin{equation}\label{ex.lin}
  a(w,\phi) = \eta F(\phi) \quad\mbox{for all }\phi\in H^m(\T^d)^2,
\end{equation}
where
\begin{align*}
  a(w,\phi) &= \int_{\T^d}(D^m w^k\cdot D^m\phi + w^k\cdot\phi)dx
	+ \int_{\T^d}\na\phi: B(\widehat w)\na w^k dx, \\
	F(\phi) &= -\frac{1}{\tau}\int_{\T^d}(\widehat u-u^{k-1})\cdot\phi dx
	+  \sum_{i=1}^2\mu_i\int_{\T^d}\widehat u_i^k\phi_i dx.
\end{align*}
Since $\widehat w\in L^\infty(\T^d)^2$ and $h'$ is continuous
in $(0,\infty)^2$, we have $\widehat u\in L^\infty(\T^d)^2$. 
This shows that $F$ is continuous on $H^m(\T^d)$.
The bilinear form $a$ is continuous and coercive, by the generalized
Poincar\'e inequality for $H^m$ spaces \cite[Chap.~2.1.4, Formula (1.39)]{Tem97}
and the positive semidefiniteness of $B(\widehat w)$ (see Lemma \ref{lem.h2A}).
Hence, the Lax-Milgram lemma provides a unique solution $w=(w_1,w_2)\in H^m(\T^d)^2
\hookrightarrow L^\infty(\T^d)^2$ to \eqref{ex.lin}. This defines the
fixed-point operator $S:L^\infty(\T^d)^2\times[0,1]\to L^\infty(\T^d)^2$,
$S(\widehat w,\eta)=w$, where $w$ solves \eqref{ex.lin}.

It holds clearly $S(w,0)=0$. Standard arguments show that $S$
is continuous (see, e.g., the proof of Lemma 5 in \cite{Jue15}). Because
of the compact embedding $H^m(\T^d)\hookrightarrow L^\infty(\T^d)$, the mapping
$S$ is even compact. In order to apply the Leray-Schauder fixed-point theorem,
it remains to prove a uniform bound for all fixed points
of $S(\cdot,\eta)$ in $L^\infty(\T^d)^2$.

Let $w\in L^\infty(\T^d)^2$ be such a fixed point, i.e.\ a solution 
to \eqref{ex.lin} with $\widehat u$ replaced by $u:=(h')^{-1}(w)$.
The uniform bound will be a consequence of the entropy inequality. For this, we
employ the test function $w$ in \eqref{ex.lin}:
\begin{align}
  \frac{\eta}{\tau}\int_{\T^d}(u-u^{k-1})\cdot w dx
	&+ \int_{\T^d}\na w: B(w)\na w dx + \tau\int_{\T^d}\big(|D^m w|^2 + |w|^2\big)dx
	\nonumber \\
	&= \eta\sum_{i=1}^2\mu_i\int_{\T^d}u_iw_i dx. \label{ex.aux1}
\end{align}
By the convexity of $h$, it follows that
$$
  h(u)-h(u^{k-1})
	\le h'(u)\cdot(u-u^{k-1}) = (u-u^{k-1})\cdot w.
$$
Moreover, by \eqref{1.eqw} and Lemma \ref{lem.h2A}, we have
\begin{align*}
  \int_{\T^d}\na w: B(w)\na w dx 
	&= \int_{\T^d}\na u:(h''A)(u)\na u dx \\
	&\ge \kappa\int_{\T^d}\bigg(\bigg(\frac{u_1}{u_2}\bigg)^{\alpha-p}
  + \bigg(\frac{u_1}{u_2}\bigg)^{p-\alpha}\bigg)|\na u|^2 dx.
\end{align*}
Taking into account the second estimate in \eqref{2.h}, we infer that
$$
  \eta\sum_{i=1}^2\mu_i\int_{\T^d}u_iw_i dx
	= \eta\sum_{i=1}^2\mu_i\int_{\T^d}u_i\pa_i h(u) dx
  \le C_h\int_{\T^d}h(u)dx.
$$
Therefore, \eqref{ex.aux1} becomes
\begin{align}
  \frac{\eta}{\tau}\int_{\T^d}h(u)dx 
	&+ \kappa\int_{\T^d}\bigg(\bigg(\frac{u_1}{u_2}\bigg)^{\alpha-p}
  + \bigg(\frac{u_1}{u_2}\bigg)^{p-\alpha}\bigg)|\na u|^2 dx \nonumber \\
	&{}+ \tau\int_{\T^d}\big(|D^m w|^2 + |w|^2\big)dx
	\le \frac{\eta}{\tau}\int_{\T^d}h(u^{k-1})dx
	+ C_h\int_{\T^d}h(u)dx. \label{ex.aux2}
\end{align}
Choosing $\tau<1/C_h$, this shows that $w$ is uniformly bounded in $H^m(\T^d)$.
Thus, we can apply the fixed-point theorem of Leray-Schauder to conclude the 
existence of a weak solution $w^k:=w$ with $u^k=h'(w^k)$ to \eqref{ex.semi}
with $\eta=1$.

{\em Step 2: a priori estimates.} Inequality \eqref{ex.aux2} shows,
for $w=w^k$, $u=u^k$, and $\eta=1$, that
\begin{align*}
  (1-C_h\tau)\int_{\T^d}h(u^k)dx 
	&+ \kappa\tau\int_{\T^d}\bigg(\bigg(\frac{u_1^k}{u_2^k}\bigg)^{\alpha-p}
  + \bigg(\frac{u_1^k}{u_2^k}\bigg)^{p-\alpha}\bigg)|\na u^k|^2 dx \nonumber \\
	&{}+ \tau^2\int_{\T^d}\big(|D^m w^k|^2 + |w^k|^2\big)dx
	\le \int_{\T^d}h(u^{k-1})dx.	
\end{align*}
We sum \eqref{ex.aux2} for $k=1,\ldots,j$ and divide the resulting inequality
by $1-C_h\tau$ (recall that we have chosen $\tau<1/C_h$):
\begin{align*}
  \int_{\T^d}h(u^j)dx &+ \frac{\kappa\tau}{1-C_h\tau}\sum_{k=1}^j
	\int_{\T^d}\bigg(\bigg(\frac{u_1^k}{u_2^k}\bigg)^{\alpha-p}
  + \bigg(\frac{u_1^k}{u_2^k}\bigg)^{p-\alpha}\bigg)|\na u^k|^2 dx \\
	&\phantom{xx}{}+ \frac{\tau^2}{1-C_h\tau}\sum_{k=1}^j
	\int_{\T^d}\big(|D^m w^k|^2 + |w^k|^2\big)dx \\
	&\le \frac{1}{1-C_h\tau}\int_{\T^d}h(u^0)dx
	+ \frac{C_h\tau}{1-C_h\tau}\sum_{k=1}^{j-1}\int_{\T^d}h(u^k)dx.
\end{align*}
We apply the discrete Gronwall inequality \cite{Cla87} to obtain for $j\tau\le T$,
\begin{align}
  \int_{\T^d}h(u^j)dx &+ \tau\sum_{k=1}^j
	\int_{\T^d}\bigg(\bigg(\frac{u_1^k}{u_2^k}\bigg)^{\alpha-p}
  + \bigg(\frac{u_1^k}{u_2^k}\bigg)^{p-\alpha}\bigg)|\na u^k|^2 dx \nonumber \\
	&{}+ \tau^2\sum_{k=1}^j\int_{\T^d}\big(|D^m w^k|^2 + |w^k|^2\big)dx
  \le C, \label{ex.aux3}
\end{align}
where $C>0$ denotes a constant which is independent of $\tau$ (and
independent of $T$ if $\mu_i\le 0$) 
\blue{but dependent on the initial entropy $H[u^0]$.}

We define the piecewise constant functions in time
$w^{(\tau)}(x,t)=w^k(x)$ and $u^{(\tau)}(x,t)=u^k(x)$ for $x\in\T^d$ and
$t\in((k-1)\tau,k\tau]$, $k=1,\ldots,j$. Furthermore, we introduce the
shift operator $\sigma_\tau u^{(\tau)}(x,t)=u^{k-1}(x)$ for $x\in\T^d$,
$t\in((k-1)\tau,k\tau]$. With this notation, we can rewrite 
\eqref{ex.aux1} (with $\eta=1$) as
\begin{align}
  \frac{1}{\tau}\int_0^T & \int_{\T^d}(u^{(\tau)}-\sigma_\tau u^{(\tau)})
	\cdot\phi dxdt + \int_0^T\int_{\T^d}\na\phi:(h''A)(u^{(\tau)})\na u^{(\tau)}dxdt 
	\nonumber \\
  &{}+ \tau\int_0^T\int_{\T^d}\big(D^m w^{(\tau)}\cdot D^m\phi 
	+ w^{(\tau)}\cdot\phi^{(\tau)}\big)dxdt 
	+ \sum_{i=1}^2\mu_i\int_0^T\int_{\T^d}u_i^{(\tau)}\phi_i dxdt \label{ex.weak}
\end{align}
and \eqref{ex.aux3} as
\begin{align}
  \int_{\T^d}h(u^{(\tau)}(t))dx
	&+ \int_0^t\int_{\T^d}\bigg(\bigg(\frac{u_1^{(\tau)}}{u_2^{(\tau)}}\bigg)^{\alpha-p}
  + \bigg(\frac{u_1^{(\tau)}}{u_2^{(\tau)}}\bigg)^{p-\alpha}\bigg)
	|\na u^{(\tau)}|^2 dxds \nonumber \\
	&{}+ \tau\int_0^t\int_{\T^d}\big(|D^m w^{(\tau)}|^2 + |w^{(\tau)}|^2\big)dxds
	\le C, \label{ex.ei}
\end{align}
where $t\in((j-1)\tau,j\tau]$. It follows that
\begin{equation}\label{est.w}
  \|w^{(\tau)}\|_{L^2(0,T;H^m(\T^d))} \le C\tau^{-1/2}.
\end{equation}
By Lemma \ref{lem.estal}, Lemma \ref{lem.h}, and estimate \eqref{ex.ei}, 
we find that
\begin{align}
  & \int_{\T^d}\bigg(\bigg|a\bigg(\frac{u_1^{(\tau)}}{u_2^{(\tau)}}\bigg)
	u_1^{(\tau)}\bigg|^2 + \bigg|a\bigg(\frac{u_1^{(\tau)}}{u_2^{(\tau)}}\bigg)
	u_2^{(\tau)}\bigg|^2 \bigg)dx
	\le 3C_a\int_{\T^d}h(u^{(\tau)})dx \le C, \label{ex.a1} \\
	& \int_{\T^d}\big((u_1^{(\tau)})^2 + (u_2^{(\tau)})^2\big)dx
	\le \int_{\T^d}h(u^{(\tau)})dx \le C. \label{ex.a2}
\end{align}
Moreover, using Lemma \ref{lem.estA} and \eqref{ex.ei},
\begin{align}
  \int_0^T\int_{\T^d} 
	& \Big(\big|\na \big(a(u_1^{(\tau)}/u_2^{(\tau)})u_1^{(\tau)}\big)\big|^2
	+ \big|\na\big(a(u_1^{(\tau)}/u_2^{(\tau)})u_2^{(\tau)}\big)\big|^2\Big)dxdt 
	\nonumber \\
	&= \int_{\T^d}|A(u^{(\tau)})\na u^{(\tau)}|^2 dx
	\le \int_0^T\int_{\T^d}|A(u^{(\tau)})|^2|\na u^{(\tau)}|^2 dxdt \nonumber \\
	&\le C_A\int_0^T\int_{\T^d}\bigg(1+\bigg(\frac{u_1^{(\tau)}}{u_2^{(\tau)}}\bigg)^4
	+\bigg(\frac{u_2^{(\tau)}}{u_1^{(\tau)}}\bigg)^4
	\bigg)|\na u^{(\tau)}|^2 dxdt \nonumber \\
	&\le C\int_0^T\int_{\T^d}\bigg(\bigg(\frac{u_1^{(\tau)}}{u_2^{(\tau)}}
	\bigg)^{\blue{\alpha-p}}
	+ \bigg(\frac{u_1^{(\tau)}}{u_2^{(\tau)}}\bigg)^{p-\alpha}\bigg)
	|\na u^{(\tau)}|^2dxdt \le C. \label{ex.a3}
\end{align}
The last but one inequality follows from the elementary estimate
$1+y^4\le y^{\alpha-p}+y^{p-\alpha}$ for $y>0$ which holds because of the
assumption $\alpha-p\ge 4$. Estimates \eqref{ex.a1}-\eqref{ex.a3} 
yield for $i=1,2$,
\begin{align}
  \|u_i^{(\tau)}\|_{L^\infty(0,T;L^2(\T^d))} +
	\|\na u_i^{(\tau)}\|_{L^2(0,T;L^2(\T^d))} &\le C, \label{est.uL2} \\
  \|a(u_1^{(\tau)}/u_2^{(\tau)})u_i\|_{L^\infty(0,T;L^2(\T^d))} 
	+ \|\na(a(u_1^{(\tau)}/u_2^{(\tau)})u_i^{(\tau)})\|_{L^2(0,T;L^2(\T^d))} &\le C.
	\label{est.aL2}
\end{align}
These estimates are uniform in $T>0$ if $\mu_i\le 0$.

Next, we derive a uniform estimate for the discrete time derivative
$(u^{(\tau)}-\sigma_\tau u^{(\tau)})/\tau$. For $\phi\in L^2(0,T;H^m(\T^d))$,
we estimate
\begin{align*}
  \frac{1}{\tau} \bigg|\int_0^T(u^{(\tau)}-\sigma_\tau u^{(\tau)})\cdot\phi
	dxdt\bigg|
	&\le \|A(u^{(\tau)})\na u^{(\tau)}\|_{L^2(0,T;L^2(\T^d))}
	\|\na\phi\|_{L^2(0,T;L^2(\T^d))} \\
	&\phantom{xx}{}+ \tau\|w^{(\tau)}\|_{L^2(0,T;H^m(\T^d))}
	\|\phi\|_{L^2(0,T;H^{m}(\T^d))} \\
	&\phantom{xx}{}+ \max\{\mu_1,\mu_2\}\|u^{(\tau)}\|_{L^2(0,T;L^2(\T^d))}
	\|\phi\|_{L^2(0,T;L^2(\T^d))} \\
	&\le C\|\phi\|_{L^2(0,T;H^m(\T^d))},
\end{align*}
taking into account the bounds \eqref{est.w}, \eqref{ex.a3}, and \eqref{est.uL2}. 
Therefore,
\begin{equation}\label{est.ut}
  \tau^{-1}\|u^{(\tau)}-\sigma_\tau u^{(\tau)}\|_{L^2(0,T;H^{m}(\T^d)')} \le C.
\end{equation}

{\em Step 3: limit $\tau \to 0$.}
Estimates \eqref{est.uL2} and \eqref{est.ut} allow us to apply the
Aubin-Lions lemma in the discrete version of \cite{DrJu12} to obtain
the existence of a subsequence, which is not relabeled, such that, as $\tau\to 0$,
$$
  u_i^{(\tau)} \to u_i \quad\mbox{strongly in }L^2(0,T;L^2(\T^d))
	\mbox{ and a.e.}, \ i=1,2.
$$
Moreover, by \eqref{est.w}, \eqref{est.uL2}, and \eqref{est.ut}, 
for the same subsequence and $i=1,2$,
\begin{align*}
 \blue{\tau w_i^{(\tau)}\to 0} &\quad\mbox{strongly in }L^2(0,T;H^m(\T^d)), \\
  \na u_i^{(\tau)}\rightharpoonup \na u_i &\quad\mbox{weakly in }L^2(0,T;L^2(\T^d)), \\
	\tau^{-1}(u_i^{(\tau)}-\sigma_\tau u_i^{(\tau)})\rightharpoonup \pa_t u_i
	&\quad\mbox{weakly in }L^2(0,T;H^{m}(\T^d)').
\end{align*}
The pointwise convergence of $(u_i^{(\tau)})$, Fatou's lemma, and estimate
\eqref{ex.ei} imply that, for a.e. $t\in(0,T)$,
\begin{align*}
  \sum_{i=1}^2\int_{\T^d}(u_i(t)-\log u_i(t))dx
	&\le \liminf_{\tau\to 0}\sum_{i=1}^2
	\int_{\T^d}\big(u_i^{(\tau)}(t)-\log u_i^{(\tau)}(t)\big)dx \\
	&\le \liminf_{\tau\to 0}\int_{\T^d}h(u^{(\tau)}(t))dx \le C.
\end{align*}
This means that $u_i>0$ a.e.\ in $\T^d\times(0,T)$.

Estimate \eqref{ex.a1} and \eqref{ex.a3} show that, up to a subsequence,
$$
  a(u_1^{(\tau)}/u_2^{(\tau)})u_i^{(\tau)} \rightharpoonup q_i
	\quad\mbox{weakly in }L^2(0,T;H^1(\T^d)),\ i=1,2,
$$
where $q_i\in L^2(0,T;H^1(\T^d))$. We wish to identify $q_i$. To this end,
let us define $\chi_\eps^{(\tau)}=\mathbf{1}_{\{u_1^{(\tau)}\ge\eps,\,
u_2^{(\tau)}\ge\eps\}}$ and $\chi_\eps=\mathbf{1}_{\{u_1\ge\eps,\,u_2\ge\eps\}}$,
where $\mathbf{1}_A$ denotes the characteristic function on the set $A$.
Clearly, $\chi_\eps^{(\tau)}\to\chi_\eps$ strongly in
$L^s(0,T;L^s(\T^d))$ for all $1\le s<\infty$. We infer that
$$
  \chi_\eps^{(\tau)}a(u_1^{(\tau)}/u_2^{(\tau)})u_i^{(\tau)}
	\rightharpoonup \chi_\eps a(u_1/u_2)u_i \quad\mbox{weakly in }L^s(0,T;L^s(\T^d)),
	\ 1\le s<2.
$$
We deduce that $q_i=a(u_1/u_2)u_i$ on the set $\{u_1\ge\eps,\,u_2\ge\eps\}$.
Since $\eps>0$ is arbitrary and $u_i>0$ a.e.\ in $\T^d\times(0,T)$, this
identification holds, in fact, \blue{a.e.\ }in $\T^d\times(0,T)$.

Consequently, we may perform the limit $\tau\to 0$ in \eqref{ex.weak}
to deduce that $u$ is a weak solution to \eqref{1.eq} \blue{with test functions} 
$L^2(0,T;H^{m}(\T^d)')$.
However, since $a(u_1/u_2)u_i\in L^2(0,T;H^1(\T^d))$, 
we employ a density argument to infer that \eqref{1.eq} also holds for
$L^2(0,T;H^{1}(\T^d)')$. Since $u_i\in L^2(0,T;H^1(\T^d))$ and
$\pa_t u_i\in L^2(0,T;H^{1}(\T^d)')$, it follows that $u_i\in C^0([0,T];L^2(\T^d))$,
so the initial datum is satisfied in $L^2(\T^d)$. Finally, since the bounds
are uniform in $T$ if $\mu_i\le 0$, the statement \eqref{1.unif} follows.


\section{Proof of Theorem \ref{thm.time}}\label{sec.time}

\blue{
Theorem~\ref{thm.ex} allows us to employ the test equations $u_1-\bar{u}_1$, 
$u_2-\bar{u}_2$ in \eqref{1.eq}, respectively:
$$
  \frac{d}{dt}\int_{\T^d}\sum_{i=1}^2 (u_i-\bar{u}_i)^2 dx 
	= -\int_{\T^d}\sum_{i=1}^2\na u_i\cdot\na (a(u)u_i) dx,
$$
which, together with Theorem~\ref{thm.ex}, implies that 
$(d/dt)\int\sum_{i=1}^2 (u_i-\bar{u}_i)^2 dx \in L^1(0,\infty)$.
Consequently, the limit
$$
  \lim_{t\to\infty}\int_{\T^d}\sum_{i=1}^2 ( u_i(t)-\bar{u}_i)^2 dx  
  = \int_{\T^d}\sum_{i=1}^2 (u_i^0-\bar{u}_i)^2 dx 
  + \int_0^\infty \frac{d}{dt}\int_{\T^d}\sum_{i=1}^2 ( u_i-\bar{u}_i)^2 dx\, dt 
$$
exists and is finite. Poincar\'e's inequality and Theorem \ref{thm.ex} imply that
$$ 
  \int_{\T^d}\sum_{i=1}^2 ( u_i-\bar{u}_i)^2 dx 
	\leq C_P \int_{\T^d}\sum_{i=1}^2 |\na u_i|^2 dx \in L^1(0,\infty), 
$$
which means that $\lim_{t\to\infty}\int_{\T^d}\sum_{i=1}^2 ( u_i(t)-\bar{u}_i)^2 dx 
= 0$. This finishes the proof.
}

\begin{remark}\label{rem.conv}\rm
If $\mu_i<0$ for $i=1,2$, we can prove the exponential convergence 
of the solution $u(t)$ to \eqref{1.eq} in $H^1(\T^d)'$ 
by using the dual method. Indeed,
let $\phi_i\in L^2(0,T;H^1(\T^d))$ be the unique solution to $-\Delta\phi_i=u_i(t)$
in $\T^d$ and $\int_{\T^d}\phi_idx=0$, $i=1,2$. 
Employing $\phi=(\phi_1,\phi_2)$ as a test function
in \eqref{1.eq}, we find after a straightforward computation that
$$
  \frac12\frac{d}{dt}\int_{\T^d}\big(|\na\phi_1|^2+|\na\phi_2|^2\big) dx 
	+ \int_{\T^d}a(u_1/u_2)(u_1^2+u_2^2)dx
	= \int_{\T^d}\big(\mu_1|\na\phi_1|^2+\mu_2|\na\phi_2|^2\big)dx.
$$
Then Gronwall's lemma implies that
$$
  \int_{\T^d}|\na\phi(t)|^2 dx 
	\le e^{\max\{\mu_1,\mu_2\}t}\int_{\T^d}|\na\phi(0)|^2dx, \quad t>0.
$$
Since $\|u_i\|_{H^1(\T^d)'}=\|\phi_i\|_{H^1(\T^d)}$, we conclude that
$\|u_i(t)\|_{H^1(\T^d)'}\le C\exp(-\kappa t)$ for $t>0$, where
$\kappa=-\max\{\mu_1,\mu_2\}>0$ and $C>0$ depends on $u^0$.
\qed
\end{remark}

\begin{remark}\label{rem.conv2}\rm
In the case $\mu_i>0$ for $i=1,2$, we cannot expect equilibration rates,
since the solution grows in the $L^2$ norm as $t\to\infty$. 
This growth can be made precise
if $\mu:=\mu_1=\mu_1>0$. Indeed, $u_i^*=e^{-\mu t}u_i$ solves
$$
  \pa_t u_i^* = \Delta\big(a(u_1^*/u_2^*)u_i^*\big), \quad t>0, \quad
	u^*_i(0) = u_i^0\quad\mbox{in }\T^d,\ i=1,2,
$$
and Theorem \ref{thm.time} shows that $u_i^*(t)\to\overline{u}_i$
in $L^2(\T^d)$ as $t\to\infty$, which translates to
$\|e^{-\mu t}u_i(t)-\overline{u}_i\|_{L^2(\T^d)}\to 0$.
\qed
\end{remark}


\end{document}